\documentclass[12pt]{amsart}
\usepackage{amsmath,amsfonts,amssymb,latexsym,amsthm}
\usepackage{epsfig}
\usepackage{hyperref} 
\usepackage[dvipsnames]{xcolor}

\numberwithin{equation}{section}

\newtheorem{thm}{Theorem}[section]
\newtheorem{lemma}[thm]{Lemma}

\newtheorem{conj}[thm]{Conjecture}
\newtheorem{conv}[thm]{Convention}

\theoremstyle{definition}

\newtheorem{defin}[thm]{Definition}

\newcommand{\sse}{\mathsf{E}}
\newcommand{\ssd}{\mathsf{D}}
\newcommand{\ssa}{\mathsf{A}}

\title[On combinatorics of Voronoi polytopes]
{On combinatorics of Voronoi polytopes for pertubations of the dual root lattices}

\author{Alexey Garber}

\thanks{\today}

\thanks{\thinspace School of Mathematical \& Statistical Sciences,
University of Texas Rio Grande Valley, Brownsville, TX, 78520}

\thanks{\thinspace {\hspace{.45ex}}
Email:
\hskip.06cm
\texttt{alexey.garber@utrgv.edu}}

\def\zz{\mathbb Z}

\def\rr{\mathbb R}

\def\<{\langle}
\def\>{\rangle}

\def\0{{\mathbf 0}}

\def\conv{{\textrm{conv}}}

\def\.{\hskip.06cm}

\begin{document}





\begin{abstract}
The Voronoi conjecture on parallelohedra claims that for every convex polytope $P$ that tiles Euclidean $d$-dimensional space with translations there exists a $d$-dimensional lattice such that $P$ and the Voronoi polytope of this lattice are affinely equivalent. The Voronoi conjecture is still open for the general case but it is known that some combinatorial restriction for the face structure of $P$ ensure that the Voronoi conjecture holds for $P$.

In this paper we prove that if $P$ is the Voronoi polytope of one of the dual root lattices $\ssd_d^*$, $\sse_6^*$, $\sse_7^*$ or $\sse_8^*=\sse_8$ or their small perturbations, then every parallelohedron combinatorially equivalent to $P$ in strong sense satisfies the Voronoi conjecture. 
\end{abstract}

\maketitle

\section{Introduction}\label{sec:intro}

Root systems appear in many mathematical fields. The associated root lattices and their dual play a prominent role in many geometric questions for lattices including sphere packing and covering problems especially in low dimensions.

In this paper we turn our attention to the properties of the Delone decompositions for dual root lattices and how the combinatorics of the corresponding tilings in their subdivisions can enforce the Voronoi conjecture on parallelohedra that establishes connection between convex polytopes that tile space with translations, parallelohedra, and Voronoi polytopes for lattices. The Voronoi conjecture \cite{Vor} claims that every convex polytope that tiles $d$-dimensional Euclidean space with translations only can be obtained as an affine image of a Voronoi polytope for $d$-dimensional lattice. 

The Voronoi conjecture is proved for $d\leq 5$ and for several families of parallelohedra with local or global combinatorial restrictions; we refer to \cite{GM} for more details and references. Particularly, the Voronoi conjecture in $\rr^5$ was proved only recently in \cite{GM}, and this proof relied on the reduction of the Voronoi conjecture to its combinatorial version from \cite{DGM} as well as on a detailed combinatorial analysis of local structure of possible five-dimensional parallelohedra tilings.

In this paper we study further avenues where a similar combinatoial reduction can be used. Our main result shows that it can be used for many parallelohedra associated with dual root lattices or their small perturbations. In the concluding sections we discuss more general appoaches for such reduction.

Before stating our main results we introduce two main notions that are needed. More detailed introduction is given in Section \ref{sec:vor}.

Given two $d$-dimensional lattices $\Lambda$ and $\Lambda'$, let $\mathcal D$ and $\mathcal D'$ be the associated Delone decompositions. We will say that $\mathcal D'$ is a {\it Delone subdivision} of $\mathcal D$ if there is an affine transformation $\mathcal A$ such that $\mathcal A(\Lambda ')=\Lambda$ and for every polytope $P$ of $\mathcal D'$, $\mathcal A(P)$ is contained in some polytope of $\mathcal D$.

In other words, we can find two affine transformations of $\mathcal D$ and $\mathcal D'$ such that the images share the vertex set and the image of $\mathcal D'$ is a subdivision of the image of $\mathcal D$. 

Now let $P$ be a parallelohedron. It is known that $P$ is centrally symmetric and that there exists unique facet-to-face tiling of the ambient space with translated copies of $P$ assuming $P$ is centered at the origin, see \cite{HSS}; let $\mathcal T_P$ denote this face-to-face tiling. For every face $F$ of this tiling we can construct the {\it dual cell} $\mathcal D(F)$ that consists of centers of all copies of $P$ that are incident to $F$.

The set of all dual cells form the {\it dual cell complex} for $P$ with the face lattice which is dual to the face lattice of $\mathcal T_P$. In case $P$ is the Voronoi polytope for some lattice, the dual cell complex geometrically coincides with the Delone decomposition of that lattice. For an arbitrary parallelohedron it only carries local combinatorics of the corresponding parallelohedral tiling.

Our main result claims that combinatorics of any Delone subdivision for some dual root lattice except $\zz^d$ is enough to enforce the Voronoi conjecture for any associated parallelohedron. In terms of dual complexes this can be formulated as follows.

\begin{thm}\label{thm:main}
Suppose $P$ is a parallelohedron such that the dual cell complex of $P$ is strongly combinatorially equivalent to a Delone subdivision of one of the lattices $\ssd_d^*$, $\sse_6^*$, $\sse_7^*$, or $\sse_8^*$, then $P$ satisfies the Voronoi conjecture.
\end{thm}

More direct formulation that requires more involved operations for the lattices is the following.

\begin{thm}\label{thm:main2} Let $P$ be a parallelohedron. If there exists a lattice $\Lambda$, such that
\begin{enumerate}
\item the Dirichlet-Voronoi polytope of $\Lambda$ is strongly combinatorially equivalent to $P$, and
\item the Delone decomposition of $\Lambda$ is strongly combinatorially equivalent to a Delone subdivision of one of the lattices  $\ssd_d^*$, $\sse_6^*$, $\sse_7^*$, or $\sse_8^*$,
\end{enumerate}
then $P$ satisfies the Voronoi conjecture.
\label{t:main}
\end{thm}

Here strong equivalence means that not only the polytopes (or decompositions) are equivalent, but that this equivalence is respected by lattice translations. We refer to \cite[Def. 1.1]{DGM} for the precise definition in case of parallelohedra.

Another point of view is the following. If $\Lambda$ is a small perturbation of one of the lattices  $\ssd_d^*$, $\sse_6^*$, $\sse_7^*$, or $\sse_8^*$, then every parallelohedron strongly combinatorially equivalent to the Voronoi polytope of $\Lambda$ satisfies the Voronoi conjecture. However the bounds for ``small'' perturbation of each lattice basis in this description should be specified separately.

Note that we do not mention the dual lattice $\ssa_{d}^*$. The Delone decomposition for $\ssa_{d}^*$ is a triangulation and therefore there is no non-trivial subdivision. For parallelohedra with dual complexes represented by triangulations the Voronoi conjecture was established by Voronoi \cite{Vor} so this case is trivial.

The paper is organized as follows. In Section \ref{sec:vor} we introduce the main notions for parallelohedra tilings and for the associated dual cell complexes as well as give a short overview of the results related to the Voronoi conjecture that we use. In Sections \ref{sec:perturbations} and \ref{sec:notations} we briefly explain how perturbations of lattices are connected with Delone subdivisions and introduce the notations needed for subsequent sections. In Sections \ref{sec:e} through \ref{sec:d_odd} we describe the Delone decompositions of dual root lattices, mostly utilizing the approach by Conway and Sloane \cite{CS91}, and establish the properties of these decompositions needed for the main results.

In Section \ref{sec:main} we prove main results. In Section \ref{sec:conclusion} we give concluding remarks on similar approach to root lattices and general lattices.

\medskip

\section{Lattices, parallelohedra, and the Voronoi conjecture}\label{sec:vor}

In this section we introduce the definitions and provide necessary background for lattices, parallelohedra, and the Voronoi conjecture.

\begin{defin}
Let $\mathbf e_1,\ldots,\mathbf e_d$ be a basis in $\rr^d$. The set of all integer linear combinations of these vectors is called a {\it lattice} $\Lambda$, i.e.
$$\Lambda:=\{x_1\mathbf e_1+\ldots x_d\mathbf e_d|x_i\in \zz\}.$$

For given lattice $\Lambda$, the {\it Voronoi polytope} of $\Lambda$ is the polytope that consists of the points that are closer to the origin than to any other point of $\Lambda$. Copies of this polytope centered at all points of $\Lambda$ form the {\it Voronoi tiling} of $\Lambda$.

The {\it Delone decomposition} of $\Lambda$ is the tiling which is geometrically dual to the Voronoi tiling of $\Lambda$. A polytope $P$ is a full-dimensional polytope of the Delone decomposition of $\Lambda$ if and only if all vertices of $P$ are in $\Lambda$, $P$ is inscribed in a sphere, and the sphere circumscribed around $P$ does not contain other points of $\Lambda$ inside or on the boundary. 
\end{defin}

\begin{defin}
Convex $d$-polytope is called a {\it parallelohedron} if it tiles $\rr^d$ with translations. 
\end{defin}

The following Minkowski-Venkov conditions \cite{Min,Ven} are necessary and sufficient for convex polytope $P$ to be a parallelohedron.
\begin{enumerate}
\item $P$ is centrally symmetric;
\item Every facet of $P$ is centrally symmetric;
\item For every codimension 2 face $F$ of $P$, the projection of $P$ along $F$ is either parallelogram or centrally symmetric hexagon.
\end{enumerate}

Moreover, if a convex polytope $P$ satisfies these conditions, then there is a face-to-face tiling with translated copies of $P$, see \cite{Ven,McM1,McM2}. This face-to-face tiling is unique assuming that one copy is centered at the origin and in this case the centers of all tiles form a $d$-dimensional lattice. For a given parallelohedron $P$ we will denote this tiling $\mathcal T_P$.

The following conjecture was formulated by Voronoi in \cite{Vor}.

\begin{conj}
For every $d$-dimensional parallelohedron $P$ there exists a $d$-dimensional lattice $\Lambda$ such that the Voronoi polytope of $\Lambda$ and $P$ are affinely equivalent.
{\sloppy

}
\end{conj} 

While there are several families of parallelohedra that are known to satisfy the Voronoi conjecture, and we again refer to \cite{GM} for more detailed survey of results, here we will concentrate on the cases that are related to the following notion of dual cells.

\begin{defin}
Let $F$ be a face of the tiling $\mathcal T_P$. The {\it dual cell} $\mathcal D(F)$ of $F$ is the collection of centers of copies of $P$ incident to $F$. 

If $F$ is a face of codimension $k$, then $\mathcal D(F)$ is called {\it dual $k$-cell}.
\end{defin}

If $P$ is the Voronoi polytope of a lattice, then the associated dual cells are the vertex sets of the faces of Delone polytopes of that lattice. However, for general parallelohedra such geometric description is not known. Nevertheless, the set of all dual cells carries the structure of a cell complex, the {\it dual cell complex}, with the face lattice dual to the face lattice of the tiling $\mathcal T_P$. 

In the view of Voronoi conjecture, the possible dual $k$-cells are expected to coincide with $k$-dimensional Delone polytopes for various lattices. This property is established only for $k\leq 3$ and is not known general. We will consider only dual 2- and 3-cells and we will refer to such cells as to 2- and 3-dimensional polytopes. Particularly, each dual 2-cell is either triangle or a parallelogram and this follows from the Minkowski-Venkov conditions.

The classification of dual 3-cell was obtained by Delone \cite{Del}, see also \cite{Mag3cells}. Each dual 3-cell belongs to one of the following geometric types.
\begin{itemize}
\item Tetrahedron;
\item Octahedron;
\item Pyramid over parallelogram;
\item Triangular prism;
\item Parallelepiped.
\end{itemize}

The following cases of the Voronoi conjecture are known for various restrictions on dual cells.

\begin{thm}[Voronoi {\cite{Vor}}]\label{thm:voronoi}
If all dual $d$-cells of a $d$-dimensional parallelohedron $P$ are simplices, then $P$ satisfies the Voronoi conjecture.
\end{thm}

\begin{thm}[Zhitomirski, {\cite{Zhi}}]\label{thm:zhitomirski}
If all dual 2-cells of a parallelohedron $P$ are triangles, then $P$ satisfies the Voronoi conjecture.
\end{thm}

\begin{thm}[Ordine, {\cite{Ord}}]\label{thm:ordine}
If each dual 3-cell of a parallelohedron $P$ is either a tetrahedron, octahedron, or pyramid, then $P$ satisfies the Voronoi conjecture.
\end{thm}

The three theorems above state that if certain faces of parallelohedron satisfy some local combinatorial condition, then this parallelohedron satisfies the Voronoi conjecture. In paper \cite{GGM}, it was shown that a certain global combinatorial condition implies the Voronoi conjecture as well. We will introduce this condition in terms of the red Venkov graph as described in \cite{Gar4dim} and \cite{DGM}.

\begin{defin}
For a fixed parallelohedron $P$ we define the following graph $G(P)$, the {\it red Venkov graph} for $P$.

The vertices of $G(P)$ are identified with equivalence classes of dual 1-cells for $P$; two 1-cells are equivalent if they differ by translation on a vector from the lattice associated with $P$.

Two distinct vertices $x$ and $y$ of $G(P)$ are connected with an edge if and only if there is a triangular dual 2-cell of $P$ that is incident to two dual 1-cells equivalent to (1-cells associated with) $x$ and $y$.

If $\Lambda$ is a lattice and $P$ is the Voronoi parallelohedron for $\Lambda$, then we also refer to $G(P)$ as the red Venkov graph for $\Lambda$ and denote it as $G(\Lambda)$.
\end{defin}


The graph $G(P)$ encodes which pairs facets of $P$ share a primitive face of codimension 2, i.e. a face with triangular dual cell. A similar approach can be used to construct the full Venkov graph that tracks both types of possible shared faces of codimension 2, see \cite{Ord}. However, for our purposes the red Venkov graph is more useful as the structure of cycles of this graph can be used to guarantee the Voronoi conjecture using only combinatorics of $P$.

\begin{defin}
Suppose vertices $x$, $y$, and $z$ of $G(P)$ correspond to edges of one triangular dual 2-cell of $P$. Then the cycle $xyzx$ of $G(P)$ is called {\it half-belt cycle}.

Let $D$ be a dual 3-cell of $P$. Suppose that the origin belongs to $D$ and that all dual 2-cells of $D$ that contain the origin are triangles. This means that $D$ is either a tetrahedron, octahedron, or pyramid and in the latter case the origin is the apex of the pyramid. Let $C$ be the cycle of $G(P)$ that consists of vertices and edges corresponding to the sequence of dual 1-cells of $D$ around the origin. Then $C$ is called {\it trivially contractible cycle}. 

All together these families of cycles are called the {\it basic cycles} of $G(P)$. We also refer to Figure \ref{pict:basic} for a visualized description of both types of cycles.

\begin{center}
\begin{figure}
\includegraphics[width=0.9\textwidth]{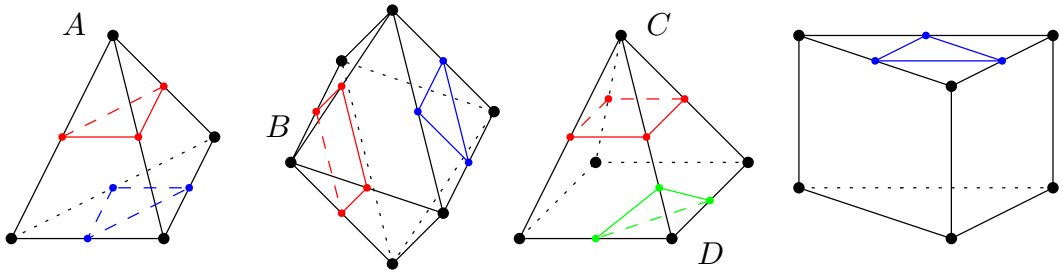}
\caption{Some of possible basic cycles. The black polytopes represent possible dual 3-cells except the cubic cell. The colored dots correspond to dual 1-cells and therefore the edges of the red Venkov graph. The red cycles show trivially contractible cycles around vertices $A$, $B$, and $C$ of the corresponding dual 3-cells. These cycles need to be understood as cycles of the corresponding dual 1-cells. Particularly, all dual 2-cells incident to $A$, $B$, and $C$ in these dual 3-cells are triangles. The blue cycles show selected half-belt cycles; each of these cycles include three dual 1-cells within one triangular dual 2-cell. The green cycle around vertex $D$ is not a trivially contractible cycle because one of the dual 2-cells incident to $D$ in the given pyramid is a parallelogram\protect\footnotemark.}
\label{pict:basic}
\end{figure}
\end{center}
\end{defin}

\footnotetext{Actually, there will be no edge associated with the dashed green edge in the corresponding Venkov graph because the two green vertices in the base of the pyramid do not correspond to dual 1-cell within one triangular dual 2-cell.}

We will use the following adaptation of \cite[Thm. 4.6]{GGM}. It was also reformulated in a similar way in \cite{Gar4dim} and \cite{DGM}.

\begin{thm}\label{thm:basic}
Let $\mathcal G$ be the group of cycles of $G(P)$ with rational coefficients\footnote{That is $\mathcal{G}$ is the group of one-dimensional rational homologies if we treat $G(P)$ as simplicial complex.}. If $\mathcal G$ is generated by the set of basic cycles, then the Voronoi conjecture is true for $P$.
\end{thm}

Theorems \ref{thm:basic} and \ref{thm:ordine} both generalize Theorems \ref{thm:voronoi} and \ref{thm:zhitomirski}. That is, if the conditions of Theorem \ref{thm:voronoi} or Theorem \ref{thm:zhitomirski} can be applied to parallelohedron $P$, then conditions of Theorem \ref{thm:basic} and Theorem \ref{thm:ordine} can be applied to $P$ as well. While there is no similar direct dependence between Theorems \ref{thm:basic} and \ref{thm:ordine}, there is a generalization of both. This generalization is given in \cite[Thm. 5.1]{DGM}. It also uses the combinatorics of $P$ to construct another simplicial complex, the Venkov complex of $P$, and its cohomologies. 

In this paper we will not use the Venkov complex, but our proof of Theorem \ref{thm:main} will use that the corresponding parallelohedra satisfy conditions of either Theorem \ref{thm:ordine} or Theorem \ref{thm:basic}. Consequently, the analogue of \cite[Thm. 5.1]{DGM} holds for such parallelohedra as well. 

\section{Perturbations of lattices}\label{sec:perturbations}


In this section we show how Delone subdivisions defined in the introduction are connected with perturbations of lattices. While this observation does not participate in the proof of any results and serves only for more perceptible reformulation of Theorem \ref{thm:main}, we think it should be presented here for completeness. The approach is largely based on \cite[Sect. 2.6]{Val}.

Let $\Lambda$ be a $d$-dimensional lattice and let $G$ be the Gram matrix of some basis $\mathcal A$ of $\Lambda$. Then the Delone decomposition of $\Lambda$ is affinely equivalent to the Delone decomposition of the integer lattice $\zz^d$ with respect to the metric defined by the quadratic form $\mathbf x^tG\mathbf x$. This is easy to see once we consider the affine transformation from $\mathcal A$ to the standard basis of $\zz^d$.

In this representation, we consider the ellipsoids of the form $$ (\mathbf x-\mathbf a)^tG(\mathbf x-\mathbf a)=r^2$$ that do not contain integer points inside. Then the integer points on such ellipsoids correspond to integer combinations of the vectors from $\mathcal A$ that constitute Delone polytopes for $\Lambda$.

Now if $\Lambda'$ is another lattice such that the Delone decomposition of $\Lambda'$ is a Delone subdivision of $\Lambda$, then we can choose a basis $\mathcal A'$ of $\Lambda'$ such that the linear transformation from $\mathcal A'$ to $\mathcal A$ gives the required subdivision. Let $G'$ be the Gram matrix of $\mathcal A'$. In that case, for empty ellipsoid of $\zz^d$ with respect to $G'$, the induced set of integer points is contained in some empty ellipsoid for $G$.

Let $P'$ and $P$ be some polytopes with integer vertices such that $P'\subseteq P$ and suppose that $P'$ and $P$ correspond to Delone polytopes of $\Lambda'$ and $\Lambda$ respectively. Then there are ellipsoids
$$ (\mathbf x-\mathbf a')^tG'(\mathbf x-\mathbf a')=r'^2 \qquad \text{and} \qquad (\mathbf x-\mathbf a)^tG(\mathbf x-\mathbf a)=r^2$$ that correspond to $P'$ and $P$. That is every vertex of $P'$ satisfies the first equality while every other point of $\zz^d$ satisfies the inequality
$$(\mathbf x-\mathbf a')^tG'(\mathbf x-\mathbf a')>r'^2.$$ A similar property holds for $P$ and the second ellipsoid.

For small $\varepsilon>0$, the lattice with the Gram matrix $G+\varepsilon G'$ can be seen as a perturbation of $\Lambda$ as its basis is a perturbation of $\mathcal A$. On the other hand, if we add two equations of ellipsoids above taking the first one with coefficient $\varepsilon$, we will get an ellipsoid with the quadratic part defined by $G+\varepsilon G'$. The new ellipsoid defines the same polytope $P'$. 

Repeating the same arguments for all polytopes of the Delone decomposition for $\Lambda'$, we get that the Delone decomposition of the new lattice coincides with the one for $\Lambda'$. Thus every Delone subdivision of $\Lambda$ can be seen as a Delone decomposition for a small perturbation of $\Lambda$.

The converse statement is also true and every small enough perturbation of the Gram matrix $G$ gives a lattice that defines some Delone subdivision of $\Lambda$.

\section{Some operations on polytopes}\label{sec:notations}

In the following sections we use the following notations for standard constructions for convex polytopes. We refer to \cite{HRGZ} for more details.

Let $P$ and $Q$ be two convex polytopes in complementary subspaces of $\rr^d$ such that $P$ and $Q$ both contain the origin $\mathbf 0$ in their (relative) interiors. The origin is the only point of intersection of the ambient subspaces for $P$ and $Q$ as these subspaces are complementary. Then the polytope $P\oplus Q:=\conv(P\cup Q)$ is called the {\it free sum} of $P$ and $Q$. Another term which is used for this construction is ``join'', see \cite{CS91}, but we reserve it for the next construction.

The faces of $P\oplus Q$ (except the free sum itself) are all convex hulls of the form $\conv(F_P\cup F_Q)$ where $F_P$ and $F_Q$ are faces of $P$ and $Q$ respectively; note that the faces can be empty but cannot coincide with $P$ or $Q$. In a similar way we define the free sum of three or more convex polytopes.

For the second construction, let $P$ and $Q$ be two convex polytopes in skewed subspaces. That is, the ambient subspaces do not intersect and the associated linear subspaces have trivial intersection. Then the polytope $P*Q:=\conv(P\cup Q)$ is called the {\it join} of $P$ and $Q$. In \cite{CS91}, the term ``separated join'' is used.

The faces of $P*Q$ are all possible convex hulls $\conv(F_P\cup F_Q)$ including empty faces and faces that are equal to either $P$ or $Q$.

%
%
%
%



\section{Lattices $\mathsf{E}_6^*$, $\mathsf{E}_7^*$, and $\mathsf{E}_8$}\label{sec:e}

For the lattices $\mathsf{E}_6^*$, $\mathsf{E}_7^*$, and $\sse_8^*=\mathsf{E}_8$, our main result follows from the structure of the associated Delone decompositions. We refer to \cite{CS91} and \cite{CS} for more details on the lattices itself.

\subsection{The lattice $\sse_6^*$}

\begin{lemma}\label{lem:e6}
If $P$ is a polytope in some Delone subdivision of $\sse_6^*$, then all two-dimensional faces of $P$ are triangles.
\end{lemma}
\begin{proof}
Let $Q$ be a six-dimensional polytope of the Delone decomposition of $\sse_6^*$. Then $Q$ is a free sum of three equilateral triangles, see \cite{CS91}. 
The polytope $Q$ is a 2-neighborly polytope, that is every pair of vertices of $Q$ is connected by an edge.

If the vertex set of $P$ is a subset of the vertex set of $Q$, then $P$ is also a 2-neighborly polytope. Hence all two-dimensional faces of $P$ are triangles.
\end{proof}

\subsection{The lattice $\sse_7^*$}
\begin{lemma}\label{lem:e7}
If $P$ is a polytope in some Delone subdivision of $\sse_7^*$, then no three-dimensional face of $P$ is a triangular prism or a parallelepiped.
\end{lemma}
\begin{proof}
Let $Q$ be a seven-dimensional polytope of the Delone decomposition of $\sse_7^*$. Then $Q$ is a seven-dimensional {\it diplo-simplex}, i.e. the convex hull of a seven-dimensional simplex and its symmetric copy; see \cite{CS91}. The polytope $Q$ has 16 vertices and each vertex of $Q$ is incident to 14 edges corresponding to each other vertex of $Q$ except the opposite one.

Suppose the vertex set of $P$ is a subset of the vertex set of $Q$ and let $P'$ be a three-dimensional face of $P$. Then every vertex of $P'$ is not connected by an edge with at most one other vertex of $P'$, thus $P'$ cannot be a triangular prism or a parallelepiped.
\end{proof}

\subsection{The lattice $\sse_8$}

\begin{lemma}\label{lem:e8}
If $P$ is a polytope in some Delone subdivision of $\sse_8$, then no three-dimensional face of $P$ is a triangular prism or a parallelepiped.
\end{lemma}
\begin{proof}
Let $Q$ be an eight-dimensional polytope of the Delone decomposition of $\sse_8$. Then $Q$ is either an eight-dimensional cross-polytope or an eight-dimensional simplex; see \cite{CS91}. In the first case,  the polytope $Q$ has 16 vertices and each vertex of $Q$ is incident to 14 edges corresponding to each other vertex of $Q$ except the opposite one. In the second case, the polytope $Q$ is a 2-neighborly polytopes. In both cases we can use the same arguments as in Lemma \ref{lem:e7} for the three-dimensional faces of $P$.
\end{proof}

\medskip

\section{Lattice $\mathsf{D}_{2m}^*$}\label{sec:d_even}

Our next goal is to give a description for the red Venkov graph of $\ssd_d^*$. We start from the even case $d=2m$.

For the lattice $\ssd_{2m}^*$ (and the lattice $\ssd_{2m+1}^*$ below) we will need more detailed structure of the Delone decomposition and associated red Venkov graph. We assume that $m\geq 3$ as the case of $\ssd_4^*=\ssd_4$ can be viewed as part of the four-dimensional case which was studied in \cite{Gar4dim}.

Geometrically, the lattice $\ssd_d^*$ can be constructed as the set of integer or half-integer points where all coordintates are not integers. That is $$\ssd_d^*=\zz^d\cup(\left(\frac12,\ldots,\frac12\right)+\zz^d)$$ It can also be seen as a high-dimension analogue of the three-dimensional BCC lattice. 

For two lattice points $x$ and $y$ we will use notation $x\oplus y$ for the parity class represented by $x+y$, i.e. $$x\oplus y:=(x+y)+\ssd_{d}^*/2\ssd_d^*.$$

\subsection{The Delone decomposition of $\ssd_{2m}^*$} \label{sec:del2m} \hfill \smallskip

Every polytope of the Delone decomposition of $\ssd_{2m}^*$ is a free sum of two $m$-dimensional cubes; see \cite{CS91}. More precisely, If $\mathbf{u}$ is a point in $\rr^{2m}$ with $m$ integer coordinates and $m$ half-integer coordinates, then the sphere of radius $\sqrt{\dfrac{m}{4}}$ centered at $\mathbf{u}$ contains $2^{m+1}$ points of $\ssd_{2m}^*$ and no points of $\ssd_{2m}^*$ inside. These $2^{m+1}$ points can be obtained by either changing the $m$ integer coordinates of $\mathbf{u}$ by $\pm\frac12$ (the first cube with half-integer vertices), or by changing the $m$ half-integer coordinates of $\mathbf{u}$ by $\pm\frac12$ (the second cube with integer vertices). There are ${\displaystyle {2m-1 \choose m-1}=\frac 12{2m \choose m}}$ translational classes of such polytopes in the Delone decomposition of $\ssd_{2m}^*$ corresponding to all possible choices of $m$ half-integer coordinates of $\mathbf{u}$; complementary choices define the same translational class as the corresponding points differ by a vector from $\ssd_{2m}^*$.

In the sequel, we will denote as $P_\mathbf{u}$ the Delone polytope of the Delone decomposition of $\ssd_{2m}^*$ centered at an appropriate point $\mathbf{u}$. Particularly, coordinates of $\mathbf{u}$ will be $0$s and $\frac12$s in most of the cases; these points represent all different classes of Delone polytopes incident to the origin.


\subsection{The red Venkov graph of $\ssd_{2m}^*$} \label{sec:graph2m} \hfill \smallskip


The vertices of the red Venkov graph correspond to the edges of the Delone decomposition. Hence, the graph $G(\ssd_{2m}^*)$ has $2m+2^{2m-1}$ vertices. First $2m$ vertices correspond to edges between vertices of one cube in $P_\mathbf u$ for all relevant $\mathbf u$; each such edge connects two points that differ by a vector $\mathbf e$ with exactly one non-zero coordinate which is $1$. We call these vertices the {\it integer} vertices of $G(\ssd_{2m}^*)$ and denote each vertex as $\mathbf e$ for appropriate $\mathbf e$. 

The remaining $2^{2m-1}$ vertices correspond to edges between two different cubes in the free sums of cubes. These edges connect vertices that differ by vectors $\mathbf t$ with all coordinates $\pm \frac 12$; opposite vectors $\mathbf t$ and $-\mathbf t$ define translationally equivalent edges and therefore the same vertex of $G(\ssd_{2m}^*)$. We say that these vertices are {\it half-integer} vertices of $G(\ssd_{2m}^*)$ and denote them $\mathbf t$ for appropriate $\mathbf t$. We can associate the half-integer vertices of $G(\ssd_{2m}^*)$ with the vertices of $2^{2m}$-dimensional cube $C_{2m}:=\left[-\frac12,\frac12 \right]^{2m}$ after identification of pairs of opposite vertices.

The edges of the graph $G(\ssd_{2m}^*)$ come from triangular faces of the Delone decomposition of $\ssd_{2m}^*$. Two vertices of $G(\ssd_{2m}^*)$ are connected with an edge if and only if there is a traingular face of the Delone decomposition of $\ssd_{2m}^*$ whose two edges correspond to these vertices.

Each triangular face of a free sum of two cubes has two vertices in one cube and one vertex in the other cube. Hence no edge of $G(\ssd_{2m}^*)$ connects two integer vertices. Additionally, two half-integer vertices of $G(\ssd_{2m}^*)$ (the vertices corresponding to edges between two different cubes) may be connected with an edge only if the corresponding vectors differ in exactly one coordinate; the vector of difference corresponds to the remaining edge of the triangular face of the free sum between two vertices of one cube.

On the other hand, for every integer vertex ${\mathbf e}$ and every half-integer vertex ${\mathbf t}$, a representative of the class $\mathbf t \oplus \mathbf e$ corresponds to a half-integer vertex of $G(\ssd_{2m}^*)$. These three vertices of the graph correspond to edges of one triangle of the Delone decomposition of $\ssd_{2m}^*$.

In order to justify that, we may assume that $\mathbf e=(1,0,\ldots,0)$. Let $x$ be the point in $\mathbb{R}^{2m}$ such that its first coordinate is $\frac12$, coordinates from $2$ to $m$ coincide with the corresponding coordinates of $\mathbf t$ (or $\mathbf t\pm\mathbf e$), and the rest of coordinates are zeros. Then $x$ has exactly $m$ nonzero coordinates that are equal to $\pm \frac12$ and the polytope of the Delone decomposition of $\ssd_{2m}^*$ centered at $x$ contains the origin, the point $\mathbf e$, and a point from one of two classes $\mathbf t$ or $\mathbf t \oplus \mathbf e$. Moreover, these three points form a triangular face of this Delone polytope because there is an edge connecting the origin and $\mathbf e$ in the corresponding cube of the free sum.

Thus every integer vertex of $G(\ssd_{2m}^*)$ is connected with every half-integer vertex. And two half-integer vertices are connected if and only if the corresponding vectors differ in exactly one coordinate. Therefore the subgraph induced by the half-integer vertices can be viewed as one-dimensional skeleton of the cube $C_{2m}$ with opposite vertices identified. This subgraph can also be seen as the one-dimensional skeleton of $(2m-1)$-dimensional cube with additional edges connecting opposite vertices, but the first description will be more visual for our proof. 

The structure of the graph $G(\ssd_{2m}^*)$ is summarized in Figure \ref{pict:graph}.

\begin{center}
\begin{figure}
\includegraphics[width=0.4\textwidth]{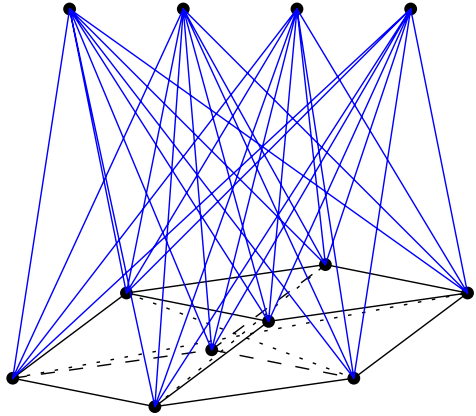}
\caption{The graph $G(\ssd_{2m}^*)$. The top level vertices are the $2m$ integer vertices. The other vertices are the $2^{2m-1}$ half-integer vertices. The blue edges connect every integer vertex with every half-integer vertex. The black edges (including dashed and dotted edges) give the structure of the 1-skeleton of the cube $C_{2m}$ with opposite vertices identified on the half-integer vertices.}
\label{pict:graph}
\end{figure}
\end{center}

\subsection{The red Venkov graph associated with a Delone subdivision of  $\ssd_{2m}^*$} \hfill \smallskip

\medskip

Suppose $\mathcal D$ is a Delone subdivision of $\ssd_{2m}^*$. That is, there is a lattice $\Lambda$ such that $\mathcal D$ is a Delone decomposition of $\Lambda$ and there exists an affine bijection from $\Lambda$ to $\ssd_{2m}^*$ that maps every polytope of $\mathcal D$ inside some free sum of two cubes within the Delone decomposition of $\ssd_{2m}^*$. In this section we describe some properties of the Venkov graph $G(\Lambda)$ associated with this Delone decomposition.

Similarly to the graph $G(\ssd_{2m}^*)$, we can split the vertices of $G(\mathcal D)$ into integer and half-integer vertices depending on whether corresponding edge connects two vertices within one cube of some free sum or two vertices from two cubes. When constructing the red Venkov graph of a subdivision, vertices and edges of the initial red Venkov graph do not disappear because edges and triangles of the initial Delone decomposition do not disappear. Hence the graph $G(\mathcal D)$ has $2^{2m-1}$ half-integer vertices.

\begin{lemma}\label{lem:i_vertex}
Let $\mathbf v$ be an integer vertex of $G(\mathcal D)$ with $k$ non-zero coordinates. Suppose $k\neq m$. Then $\mathbf v$ is connected by edges with exactly $2^{2m-k}$ half-integer vertices that can be obtained from $\frac 12\mathbf v$ by changing zero coordinates to $\pm\frac12$.
\end{lemma}
\begin{proof}
First we notice that $1\leq k\leq m$ because exactly such vertices appear in Delone polytopes of $\ssd_{2m}^*$ incident to the origin. Moreover, without loss of generality we may assume that $\mathbf v=(1^k,0^{2m-k})$, that is first $k$ coordinates of $\mathbf v$ are 1s and all other are 0s.

Let $\mathbf t$ be some half-integer vertex of $G(\mathcal D)$ connected with $\mathbf v$. There is a Delone triangle in $\mathcal D$ with vertices $z$, $z+\mathbf v$, and $z\pm\mathbf t$ for some point $z$. We may assume that $z$ is the origin and in that case these three points must be in one of the free sums from the Delone decomposition of $\ssd_{2m}^*$. The center of this free sum must have first $k$ coordinates equal to $\frac 12$. This implies that $\mathbf t$ (or $-\mathbf t$) is a half-integer vertex of the form described in the statement of the lemma.

It remains to show that all such half-integers vertices are connected with $\mathbf v$. If $k=1$, then the statement follows from Subsection \ref{sec:graph2m}. We use similar approach for other $k$.

Suppose $1<k<m$. Let $\mathbf t$ be any half-integer vertex of $G(\mathcal D)$ of the form $\mathbf t=(\frac12^k,*)$ where $*$ is placeholder for any sequence of $2m-k$ coordinates equal to $\pm\frac12$. Let $x$ be the point with first $m$ coordinates equal to the first $m$ coordinates of $\mathbf t$ and with all other coordinates being zeros. Then the free sum $P_x$ contains the origin and the points $\mathbf v$ and $\mathbf t$.

Let $F$ be the minimal face of the integer cube of $P_x$ that contains the origin and $\mathbf v$. The convex hull of $F$ and $\mathbf t$ is a face of $P_x$. In $\mathcal D$, this convex hull is also subdivided into smaller faces, and since there is an edge connecting $\mathbf v$ with the origin, there is a face $F'$ of subdivision that contains this edge and $\mathbf t$. In this case the triangle with vertices $\mathbf v$, $\mathbf t$, and the origin is a face of $F'$ and hence a Delone triangle of $\mathcal D$. This means that the vertices $\mathbf v$, $\mathbf t$, and $\mathbf v\oplus \mathbf t$, the edges of this triangle, are connected with edges in $G(\mathcal D)$.
\end{proof}

\section{Lattice $\mathsf{D}_{2m+1}^*$}\label{sec:d_odd}

\medskip


We assume that $m\geq 2$. It is worth noting that the case $m=2$ and the lattice $\ssd_5^*$ with its Delone subdivisions was consider in \cite{DGM} as part of the five-dimensional analysis. 

\subsection{The Delone decomposition of $\ssd_{2m+1}^*$} \label{sec:del2m+1} \hfill \smallskip

Every polytope of the Delone decomposition of $\ssd_{2m+1}^*$ is a join of two $m$-dimensional cubes; see \cite{CS91}. More precisely, If $\mathbf{u}$ is a point in $\rr^{2m+1}$ with $m$ integer coordinates, $m$ half-integer coordinates, and one {\it unique} coordinate in $\frac 14+\frac12\zz$, then the sphere of radius $\sqrt{\dfrac{4m+1}{16}}$ centered at $\mathbf{u}$ contains $2^{m+1}$ points of $\ssd_{2m+1}^*$ and no points of $\ssd_{2m+1}^*$ inside. These $2^{m+1}$ points can be obtained be either changing the $m$ integer coordinates of $\mathbf{u}$ by $\pm\frac12$ and setting the unique coordinate to the closest half-integer (the first cube with half-integer vertices), or by changing the $m$ half-integer coordinates of $\mathbf{u}$ by $\pm\frac12$ and setting the unique coordinate to the closest integer (the second cube with integer vertices). 

There are $(2m+1){\displaystyle {2m \choose m}=\frac 12{2m \choose m}}$ translational classes of such polytopes in the Delone decomposition of $\ssd_{2m+1}^*$ corresponding to all possible choices of the unique coordinate in $\frac 14 +\zz$ and then $m$ half-integer coordinates of $\mathbf{u}$.

In the sequel, we will denote as $P_\mathbf{u}$ the Delone polytope of the Delone decomposition of $\ssd_{2m+1}^*$ centered at an appropriate point $\mathbf{u}$. Particularly, coordinates of $\mathbf{u}$ will be $0$s, $\frac12$s and $\pm\frac14$ in most of the cases; these points represent all different classes of Delone polytopes incident to the origin.


\subsection{The red Venkov graph of $\ssd_{2m+1}^*$} \label{sec:graph2m+1} \hfill \smallskip

The structure of the red Venkov graph $G(\ssd_{2m+1}^*)$ is similar to the one described in Subsection \ref{sec:graph2m}.

The graph $G(\ssd_{2m+1}^*)$ has $2m+1+2^{2m}$ vertices. First $2m+1$ vertices correspond to edges between vertices of a single cube in $P_\mathbf u$ for all relevant $\mathbf u$; each such edge connects two points that differ by a vector $\mathbf e$ with exactly one non-zero coordinate which is $1$. We call these vertices the {\it integer} vertices of $G(\ssd_{2m+1}^*)$ and denote each vertex as $\mathbf e$ for appropriate $\mathbf e$. 

The remaining $2^{2m}$ vertices correspond to edges between two different cubes in the joins of cubes. These edges connect vertices that differ by vectors $\mathbf t$ with all entries equal to $\pm \frac 12$; opposite vectors $\mathbf t$ and $-\mathbf t$ define translationally equivalent edges and therefore the same vertex of $G(\ssd_{2m+1}^*)$. We say that these vertices are {\it half-integer} vertices of $G(\ssd_{2m+!}^*)$ and denote them $\mathbf t$ for appropriate $\mathbf t$. We can associate the half-integer vertices of $G(\ssd_{2m+1}^*)$ with the vertices of $2^{2m+1}$-dimensional cube $C_{2m+1}:=\left[-\frac12,\frac12 \right]^{2m+1}$ after identification of pairs of opposite vertices.

The edges of the graph $G(\ssd_{2m+1}^*)$ come from triangular faces of the Delone decomposition. Two vertices of $G(\ssd_{2m+1}^*)$ are connected with an edge if and only if, there is a traingular face of the Delone decomposition of $\ssd_{2m+1}^*$ whose two edges correspond to these vertices.

Each triangular face of a join of two cubes has two vertices in one cube and one vertex in the other cube. Hence no edge of $G(\ssd_{2m+1}^*)$ connects two integer vertices. Additionally, two half-integer vertices of $G(\ssd_{2m+1}^*)$ (the vertices corresponding to edges between two different cubes) may be connected with an edge only if the corresponding vectors differ in exactly one coordinate.

On the other hand, for every integer vertex ${\mathbf e}$ and every half-integer vertex ${\mathbf t}$, a representative of the class $\mathbf t \oplus \mathbf e$ corresponds to a half-integer vertex of $G(\ssd_{2m+1}^*)$. These three vertices of the graph correspond to edges of one triangle of the Delone decomposition of $\ssd_{2m+1}^*$ in the same way we established for the graph $G(\ssd_{2m}^*)$.


Thus, every integer vertex of $G(\ssd_{2m+1}^*)$ is connected with every half-integer vertex. And two half-integer vertices are connected if and only if the corresponding vectors differ in exactly one coordinate. Therefore the subgraph induced by the half-integer vertices can be viewed as one-dimensional skeleton of the cube $C_{2m+1}$ with opposite vertices identified. 

%


\subsection{The red Venkov graph associated with a Delone subdivision of  $\ssd_{2m+1}^*$} \hfill \smallskip

\medskip

Similarly to Delone subdivisions of $\ssd_{2m}^*$, if $\mathcal D$ is a Delone subdivision of $\ssd_{2m+1}^*$, then we can split the vertices of $G(\mathcal D)$ into integer and half-integer vertices depending on whether corresponding edge connects two vertices within one cube of some join or two vertices from two cubes. When constructing the red Venkov graph of a subdivision, vertices and edges of the initial red Venkov graph do not disappear because edges and triangles of the initial Delone decomposition do not disappear. Hence the graph $G(\mathcal D)$ has $2^{2m-1}$ half-integer vertices.

We also can formulate the following lemma; it is an analogue of Lemma \ref{lem:i_vertex}. The proof is also similar.


\begin{lemma}\label{lem:i_vertex_odd}
Let $\mathbf v$ be an integer vertex of $G(\mathcal D)$ with $k$ non-zero coordinates. We may assume $k\leq m$. Then $\mathbf v$ is connected by edges with exactly $2^{2m+1-k}$ half-integer vertices that can be obtained from $\frac 12\mathbf v$ by changing zero coordinates to $\pm\frac12$.
\end{lemma}

\section{Proof of the main results}\label{sec:main}

\medskip


In this section we prove the main theorem of this paper, Theorem \ref{thm:main}. We split it in several lemmas.

\begin{lemma}\label{lem:e}
Theorem \ref{thm:main} holds for lattices $\sse_6^*$, $\sse_7^*$, and $\sse_8^*$. 
\end{lemma}
\begin{proof}
Lemma \ref{lem:e6} ensures that every two-dimensional face of every Delone subdivision of $\sse_6^*$ is a triangle. Therefore the corresponding paralleloheda satisfy the Voronoi conjecture due to result of Zhitomirski \cite{Zhi}.

Lemmas \ref{lem:e7} and \ref{lem:e8} ensure that no three-dimensional face of any Delone subdivision of $\sse_7^*$ and $\sse_8^*$ is a triangular prism or a parallelepiped. Thus, the corresponding parallelohedra satisfy the Voronoi conjecture due to result of Ordine \cite{Ord}.
\end{proof}

\begin{lemma}\label{lem:d_even}
Let $m\geq 3$ be an integer. Theorem \ref{thm:main} holds for the lattice $\ssd_{2m}^*$. 
\end{lemma}

\begin{proof}
We first proof the statement for the case when the dual cell complex of $P$ is equivalent to the Delone decomposition of $\ssd_{2m}^*$. After that we give a sketch of the proof for the Delone subdivisions of $\ssd_{2m}^*$ and the complete proof of this case with all the details is given in Appendix \ref{sec:appendix}.

We need to show that the set of basic cycles (half-belt and trivially contractible cycles) generate the group of cycles of $G(\ssd_{2m}^*)$. Let $\mathcal G$ be the group of cycles of $G(\ssd_{2m}^*)$, and let $\mathcal C$ be the subgroup generated by the basic cycles. We will show that for every cycle $x$ of $\mathcal G$, the coset $x+\mathcal C$ contains a combination of trivially contractible cycles and therefore $x$ is an element of $\mathcal C$.

Suppose $x$ passes through an integer vertex $\mathbf e$ of $G(\ssd_{2m}^*)$. According to the results of Subsection \ref{sec:graph2m} there are two half-integer vertices $\mathbf a$ and $\mathbf b$ such that $x=\ldots \mathbf{aeb}\ldots$. The vertices $\mathbf a$ and $\mathbf b$ can be connected by a path of edges through half-integer vertices only. We will show that changing two edges $\mathbf{aeb}$ to this path does not change the coset of $x$.

Since $\mathbf e$ is connected by edges with every half-integer vertex, it is enough to show this property if there is an edge between $\mathbf a$ and $\mathbf b$, that is if $\mathbf a \oplus \mathbf b$ corresponds to an integer vertex of $G(\ssd_{2m}^*)$. If $\mathbf a \oplus \mathbf b=\mathbf e$, then there is a triangular face in the Delone decomposition of $\ssd_{2m}^*$ with edges corresponding to $\mathbf a$, $\mathbf b$, and $\mathbf e=\mathbf a\oplus \mathbf b$ (see Subsection \ref{sec:graph2m}). Thus the cycle $\mathbf{aeba}$ is a half-belt cycle and swapping two edges $\mathbf{aeb}$ to $\mathbf{ab}$ does not change the coset.

Now suppose $\mathbf f :=\mathbf a \oplus \mathbf b\neq \mathbf e$; then $\mathbf f$ is also an integer vector with one non-zero coordinate. In that case we can find a free sum of two cubes such that $\mathbf e$ represents a side of one cube, $\mathbf f$ represents a side of the other cube, and $\mathbf a$ and $\mathbf b$ represent two vectors between the cubes. 

Without loss of generality we may assume that $\mathbf e$ has one coordinate 1 and that $\mathbf a$ and $\mathbf b$ have  coordinates $\frac 12$ in that position. Let $y$ be any point that satisfies the following properties
\begin{itemize}
\item $y$ has coordinate $\frac12$ in the place where $\mathbf e$ has nonzero coordinate,
\item $y$ matches with both $\mathbf a$ and $\mathbf b$ in $m-1$ common coordinates, and
\item the remaining $m$ coordinates of $y$ are zeros.
\end{itemize}
Then $P_y$ is the desired free sum, see the left part of Figure \ref{pict:remove_integer}.

\begin{center}
\begin{figure}
\includegraphics[width=\textwidth]{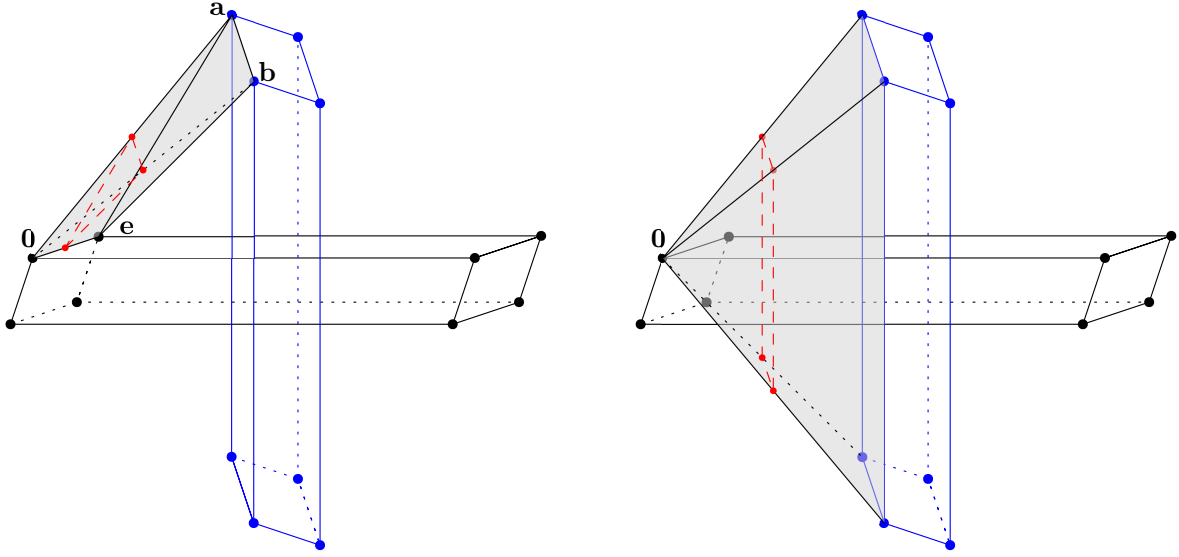}
\caption{The left part of the figure shows the tetrahedral face $\mathbf{0aeb}$ of the free sum of two cubes. The red cycle shows the trivially contractible cycle $\mathbf{aeba}$ used in removing integer vertices from the cycle $x$. The right shows a pyramidal face of the free sum. The red cycle is the trivially contractible cycle within this pyramid used to generate the boundary of a face of $C_{2m}$. The black points correspond to integer vertices of the free sum and the blue points correspond to half-integer vertices of the free sum.}
\label{pict:remove_integer}
\end{figure}
\end{center}

The four vertices of $P_y$ representing classes of the origin, $\mathbf e$, $\mathbf a$, and $\mathbf b$ form a tetrahedral face of $P_y$ and the cycle $\mathbf{aeba}$ is a trivially contractible cycle in this tetrahedron. Swapping two edges $\mathbf{aeb}$ to $\mathbf{ab}$ does not change the coset of $x$.

Repeating these steps we find a cycle $x'\in x+\mathcal C$ that contains only edges between half-integer vertices of $G(\ssd_{2m}^*)$, i.e. $x'$ can be represented as a path in the 1-skeleton of the cube $C_{2m}$ that is either a cycle or starts and ends in two opposite vertices of this cube.

We claim that the boundary of every two-dimensional face of $C_{2m}$ is trivially contractible cycle. Indeed, suppose the free sum of two cubes $P\oplus Q$ is incident to the origin and origin is a vertex of $P$. Then the origin together with one square face of $Q$ form a pyramidal face of the free sum. The trivially contractible cycle of this pyramid is exactly the boundary of a two-dimensional face of $C_{2m}$, see the right part of Figure \ref{pict:remove_integer}. It is easy to see that we can get boundary of every face of $C_{2m}$ by choosing an appropriate free sum.

The boundaries of all two-dimensional faces of $C_{2m}$ together with a path between two opposite vertices generate all cycles in $\mathcal G$ if we use integer coefficients. However a multiple of this path is generated by boundaries as well so all cycles of $\mathcal G$ are generated by the boundaries with rational coefficients. Thus $x'$ belong to the trivial coset, so is the initial cycle $x$.


Now let $\mathcal D$ be some Delone subdivision of $\ssd_{2m}^*$. For $\mathcal D$ we use a similar approach reducing every cycle of the red Venkov graph to cycle on $1$-skeleton of the cube $C_{2m}$. However, since the red Venkov graph $G(\mathcal D)$ may have additional vertices we first apply another reduction approach to deal with additional vertices and edges.

Recall that we split all vertices of $G(\mathcal D)$ in integer and half-integer vertices. We identify the following additional classes of edges and vertices.
\begin{itemize}
\item D-vertices (or diagonal vertices). These are the integer vertices of $G(\mathcal D)$ that correspond to some diagonal of a cube in some free sum. They correspond to vectors of the form $(\pm1^m,0^m)$ where exactly $m$ coordinates are zeros.
\item I-vertices (or non-diagonal integer vertices). These are all other integer vertices.
\item IH-edges (or integer-to-half-integer edges). These are the edges between one integer vertex and one half-integer vertex.
\item II-edges. These are the edges of  $G(\mathcal D)$ that connect two integer vertices.
\item D-edges (or diagonal edges between half-integer vertices). These edges are between two half-integer vertices that differ in exactly $m$ coordinates.
\item S-edges (or edges of the 1-skeleton). These are the edges between two half-integer vertices that belong to 1-skeleton of the cube $C_{2m}$.
\item H-edges (or non-diagonal and non-skeletal edges between half-integer vertices). These are all other edges between half-integer vertices.
\end{itemize}

In these terms, the idea of the proof for the lattice $\ssd_{2m}^*$ above can be reformulated as follows. If our cycle $x$ contains a pair of IH-edges with common I-vertex, then we substitute it with a sequence of S-edges and get another cycle from the coset $x+\mathcal S$. Similar steps for subdivisions of $\ssd_{2m}^*$ are the following.

Suppose $x$ is a cycle in the red Venkov graph $G(\mathcal D)$ of some Delone subdivision $\mathcal D$ of $\ssd_{2m}^*$. We will use the following modifications of $x$ without changing the corresponding coset $x+\mathcal S$ where $\mathcal S$ the subgroup generated by the basic cycles. For each step we will use the same notation $x$ for the old/new cycle.

\begin{enumerate}
\item If $x$ passes through a D-vertex, then we can change a pair of consecutive edges through such a vertex with a sequence of edges that are not incident to any D-vertex.
\item If $x$ contains II-edges, then we can change every II-edge with a pair of IH-edges without adding new integer vertices. Note that the resulting cycle will not contain D-vertices after such changes.
\item If $x$ contains D-edges, then we can swap each D-edge into two H- or S- edges, or two IH-edges. Note that after this step all edges of $x$ will be IH-, S- or H-edges, and there will be only I-vertices among integer vertices of $x$.
\item If $x$ contain H-edges, then every such edge can be swapped by a pair of IH-edges incident to only I-vertices among integer vertices. Note that at this point $x$ will contain only I-vertices and half-integer vertices and all edges of $x$ will be S-edges or IH-edges. This resembles the situation for the lattice $\ssd_{2m}^*$ with only exception that we may have more integer vertices but not D-vertices.
\item If $x$ contains I-vertices, then we can swap a pair of edges incident to one such vertex with a sequence of S-edges similarly to the approach for $\ssd_{2m}^*$ above.
\end{enumerate}

After performing these steps we get a cycle $x'\in x+\mathcal S$ that consists of S-edges only. Using a small modification of the proof for $\ssd_{2m}^*$ we can show that $x'$ is also generated by basic cycles. The details of the proof and the separate steps are given in Appendix~\ref{sec:appendix}.
\end{proof}

\begin{lemma}\label{lem:d_odd}
Let $m\geq 2$ be an integer. Theorem \ref{thm:main} holds for the lattice $\ssd_{2m+1}^*$. 
\end{lemma}
\begin{proof}
The proof is similar to the proof of Lemma \ref{lem:d_even}. We highlight the main steps and only emphasize the differences in two proofs.


We use the structure of the graph $G(\ssd_{2m+1}^*)$ described in Section \ref{sec:d_odd} and show that every pair of edges incident to one integer vertex can be changed into a sequence of edges between half-integer vertices. After that we can show that every cycle on half-integer vertices is generated by the basic cycles of $G(\ssd_{2m+1}^*)$.

The situation with Delone subdivisions is even slightly simpler for $\ssd_{2m+1}^*$ than for $\ssd_{2m}^*$. There is no need to consider D-vertices or D-edges separately as they can be treated as general I-vertices or H-edges respectively. This is because the case $k=m$ is not special for Lemma \ref{lem:i_vertex_odd} while it is special for Lemma \ref{lem:i_vertex}. Another point of view is the following, the cubes of the joins in the Delone decomposition of $\ssd_{2m+1}^*$ are faces of these joins (and hence their diagonals define faces), while this is not true in free sums of the Delone decomposition of $\ssd_{2m}^*$.

Thus the proof follows the same framework as described in Appendix \ref{sec:appendix} but with some steps being redundant. The justification for separate steps is similar for $\ssd_{2m+1}^*$.
\end{proof}

Combining Lemmas \ref{lem:e}, \ref{lem:d_even}, and \ref{lem:d_odd} we get a proof for Theorem \ref{thm:main} as well as for its reformulation  as Theorem \ref{thm:main2}.

\section{Concluding remarks}\label{sec:conclusion}

\medskip

In this paper we study the dual root lattices  and combinatorics of the Delone decompositions of their perturbations. Similar question whether all perturbations of a given lattice $\Lambda$ carry enough combinatorics to ensure the Voronoi conjecture can be asked for any $\Lambda$. In this section we briefly discuss this question for some other families of lattices.

\subsection{Root lattices} One of the most natural examples to consider probably even before dual root lattices is the root lattices themselves. However, this case appears to be more involved than the dual root lattices despite sharing the same symmetries.

We will use the lattice $\ssd_d$ and its dual to illustrate the case. For the dual lattice, the Delone polytopes are either free sums or joins of two cubes of dimension $\lfloor \frac{d}{2}\rfloor$. On the other hand, for $\ssd_d$, the Delone polytopes are either cross-polytopes or half-cubes \cite{CS91}. In the latter case the polytopes have $2^{d-1}$ vertices or about the quadratic number of the number of vertices for Delone polytopes of $\ssd_d^*$.

Thus, we expect that there  are considerably more ways to subdivide $d$-dimensional half-cubes than $d$-dimensional free sums of joins. While it does not mean that combinatorics of the perturbed lattice will not be enough for the Voronoi conjecture, the associated Venkov graph could change a lot compared to the graph $G(\ssd_d)$.

Nevertheless we believe that the analogue of Theorem \ref{thm:main} holds for all root lattices too; however, we do not have a justification at the moment.

\subsection{Rigid lattices}


Another family of lattices that worth considering for similar question is the class of rigid lattices. A lattice $\Lambda$ is called {\it rigid} if every perturbation of $\Lambda$ other than scaling has an affinely different Delone decomposition, see \cite{DG}. Among the root lattices and their dual in dimensions at least 2, the lattices $\ssd_d$, $\ssd_{2m}^*$, $\sse_6$, $\sse_6^*$, $\sse_7$, $\sse_7^*$, and $\sse_8=\sse_8^*$ are rigid.

Every $d$-dimensional lattice can be represented as a combination of rigid lattice but possibly of smaller dimensions, see \cite{DGSW} and references therein for an approach to enumeration of Voronoi parallelohedra that uses rigid lattices (or extreme rays). For example, there are seven five-dimensional rigid lattices, see \cite{BG} and \cite{DGSW}.

For 3 out of 7 of the five-dimensional rigid lattices, the approach that we used for lattices $\sse_6^*$, $\sse_7^*$, and $\sse_8$ can be used. For these three rigid lattices, every Delone polytope is either 2-neighborly or ``almost'' 2-neighborly meaning that every vertex is connected by edges with all other vertices but possibly one.

This observation does not prove any new result because the analogue of Theorem \ref{thm:main} was proved in \cite{DGM} for all five-dimensional lattices. However, this observation could make the computations in \cite{DGM} considerably faster as it shows that for some sizeable amount of lattices, the main result of \cite{DGM} can be established without computations of the associated Venkov graphs (or Venkov complexes).

Speaking about higher dimensions, a complete list of six-dimensional rigid lattices is not known. However, there is list of more than 25,000 rigid six-dimensional lattices \cite{DV} that appear ``close'' to the lattice $\sse_6^*$. These six-dimensional lattices can serve as first candidates to check a similar approach in $\rr^6$.

\subsection{Experimental avenues} Particularly, the lattices mentioned above open several experimental avenues to approach the Voronoi conjecture and a theoretical counterexample. If there is a way to construct a lattice (or the corresponding Voronoi polytope) that does not satisfy Theorem \ref{thm:basic} or its strengthening in \cite[Thm 5.1]{DGM}, then it could possibly mean that combinatorics of parallelohedra is not enough to enforce the Voronoi conjecture and further geometric arguments are needed.

Alternatively, the combinatorics of such a theoretical lattice or parallelohedron can be used to construct a counterexample to the Voronoi conjecture.

At this point, all examples that were considered do satisfy Theorem \ref{thm:basic}.

\subsection*{Acknowledgments.}

\vskip.7cm

This work was completed while the author was a visiting professor at IST Austria. The author is thankful to IST Austria and the group of Herbert Edelsbrunner for hospitality and support.

\appendix

\section{Delone subdivisions of $\ssd_{2m}^*$}\label{sec:appendix}


Here we provide the complete proof for the approach described in the proof of Lemma~\ref{lem:d_even} for all Delone subdivisions of $\ssd_{2m}^*$; similar approach can be used for Lemma \ref{lem:d_even} and some steps of our proof are redundant. We fix one such subdivision $\mathcal D$ of $\ssd_{2m}^*$ corresponding to some lattice $\Lambda$. Recall that $\mathcal G$ and $\mathcal S$ are the group of cycles of $G(\mathcal D)$ and its subgroup generated by the basic cycles.

We also fix one cycle $x$ in the associated red Venkov graph $G(\mathcal D)$. We will show that we can choose a representative from $x+\mathcal S$ in $\mathcal G$ which is generated by basic cycles itself. This is enough to complete the proof of Lemma \ref{lem:d_even}.

\subsection{The cycle $x$ passes through a D-vertex} Suppose $\mathbf a$ is a D-vertex of $x$ and let $\ell$ be the corresponfing edge of the decomposition $\mathcal D$. Note that $\ell$ belongs only to triangular faces of $\mathcal D$ because each free sum of the Delone subdivision of $\ssd_{2m}^*$ contains at most one pair of vertices differ by $\ell$.

We may assume that if $\mathbf{bac}$ are two consecutive edges of $x$, then $\mathbf b$ and $\mathbf c$ correspond to two edges within one three-dimensional polytope $P$ of $\mathcal D$. Moreover, $P$ is either tetrahedron, octahedron, or quadrangular pyramid.   

In all cases, all cycles within the subgraph of $G(\mathcal D)$ induced by the vertices corresponding to the edges of $P$ are generated by basic cycles, and we can find a path between $\mathbf b$ and $\mathbf c$ within this subgraph that does not go through $\mathbf a$. We can swap the pair of edges $\mathbf{bac}$ by this path in $x$ and this path will not contain another D-vertex. Indeed, if there is another D-vertex in the new path, then there are two edges corresponding to D-vertices in one free sum of cubes and these edges intersect in the center of the free sum which is impossible for Delone decomposition of $\Lambda$.

After performing such swaps for every D-vertex in $x$ we get a new cycle in the same coset $x+\mathcal S$ (which we will also refer as $x$) without D-vertices.

\subsection{The cycle $x$ contains II-edges} Suppose $\mathbf{ab}$ is an II-edge of $x$. Then there is a triangle of $\mathcal D$ with edges represented by integer vectors $\mathbf a$, $\mathbf b$, and $\mathbf a\oplus\mathbf b$. After proper translation we may assume that one vertex of this triangle is at the origin $\mathbf 0$, and two others are at $\mathbf a$ and $\mathbf b$. We claim that there exists a half-integer vertex $\mathbf c$ such that $\mathbf{0abc}$ is a tetrahedron of $\mathcal D$. 

Note that neither $\mathbf a$ nor $\mathbf b$ correspond to D-vertex. If $\mathbf{a}\oplus \mathbf{b}$ is not a D-vertex, then the triangle $\mathbf{0ab}$ is part of one face $F$ of one integer cube of some free sum within the Delone decomposition of $\ssd_{2m}^*$. Adding any half-integer vertex of this free sum to $F$ we get another face. The subdivision $\mathcal D$ induces a subdivision of this face and hence this additional vertex, which we denote $\mathbf{c}$, gives the desired tetrahedron.

If $\mathbf{a}\oplus \mathbf{b}$ is a D-vertex, then let $P$ be any full-dimensional polytope of $\mathcal D$ incident to $\mathbf{0ab}$ and let $F$ be any three-dimensional face of $P$ that contains $\mathbf{0ab}$ and at least one non-integer vertex. Similarly to the previous step, $F$ is either tetrahedron, octahedron, or pyramid.

If $F$ is an octahedron, then there are two edges of $F$ equivalent to $\mathbf{a}\oplus \mathbf b$, but this is impossible within one free sum. If $F$ is a pyramid, then $\mathbf{0ab}$ is one of its faces and one more vertex of $F$ is a half-integer point. This means that among two parallel sides of the base of $F$, one connects two integer points and the other one connects two half-integer points. This is again impossible within one free sum of cubes. This leaves us with the only option that $F$ is a tetrahedron $\mathbf{0abc}$.

Now in $\mathbf{0abc}$, we can swap the edge $\mathbf{ab}$ with the path $\mathbf{acb}$ because the cycle $\mathbf{abca}$ is trivially contractible. After this change we swap II-edge with two IH-edges. After doing this for all II-edges we get a representative of $x+\mathcal{S}$ without D-vertices and without II-edges as we don't add any integer vertex to $x$. 


\subsection{The cycle $x$ contains D-edges} Suppose $\mathbf{ab}$ is a D-edge of $G(\mathcal D)$. Then $\mathbf{a}\oplus \mathbf{b}$ represents a D-vertex of $G(\mathcal D)$ and there is a triangle of $\mathcal D$ with edges equivalent to $\mathbf a$, $\mathbf b$, and $\mathbf a \oplus \mathbf b$. Let $F$ be any three-dimensional face of $\mathcal D$ incident to this triangle. We claim that $F$ is a tetrahedron.

Similarly to the previous step, $F$ is either a tetrahedron or a quadrangular pyramid because one edge of $F$ corresponds to the D-vertex $\mathbf a \oplus \mathbf b$ of the red Venkov graph. However, if $F$ is a pyramid, then $\mathbf a \oplus \mathbf b$ corresponds to its side edge and two base vertices of $F$ belong to different cubes of some free sum $P$. The other two base vertices must belong to different cubes of the free sum as well.

If pairs of vertices of the base from each cube of $P$ form the sides of the base, then the two cubes of the free sum have parallel sides which is impossible. If these pairs form the diagonals, then the center of the base is the center of the free sum too as this is the only common point of two cubes. However, $\mathbf a \oplus \mathbf b$ corresponds to a D-vertex, so its midpoint must be the center of the free sum which gives a contradiction. Thus $F$ must be a tetrahedron.

We can assume that two vertices connected with the edge $\mathbf a \oplus \mathbf b$ are integer vertices of the free sum that contains $F$, and the third vertex of the initial triangle is a half-integer vertex. The fourth vertex can be either integer or half-integer, but in both cases we can use the cycle shown in Figure \ref{pict:d_edge} to change the D-edge $\mathbf{ab}$ into either two S- or H-edges (left part) or two IH-edges (right part). Note, that we do not create II-edges and we do not add D-vertices or D-edges in the process.

\begin{center}
\begin{figure}
\includegraphics[width=\textwidth]{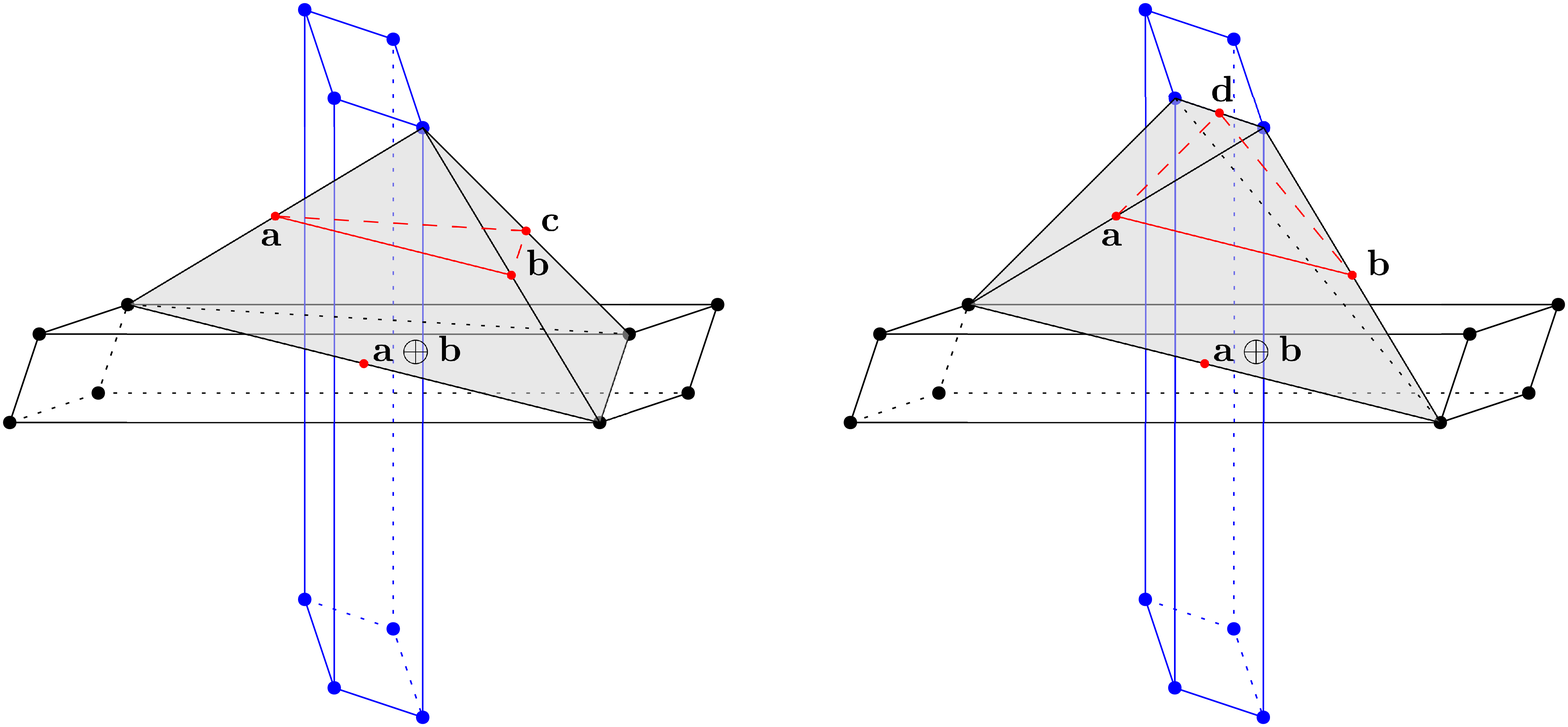}
\caption{Two possible ways to swap the D-edge $\mathbf{ab}$ in the cycle $x$. In both cases the red cycle is trivially contractible within the highlighted tetrahedral dual 3-cell. If the fourth vertex of the corresponding tetrahedron is an integer vertex (the left part), then $\mathbf{ab}$ can be swapped with two edges $\mathbf{acb}$ and the new edges are S- or H-edges because all $\mathbf a$, $\mathbf b$, and $\mathbf c$ correspond to half-integer vertices of the red Venkov graph. If the fourth vertex of the corresponding tetrahedron is a half-integer vertex (the right part), then $\mathbf{ab}$ can be swapped with two edges $\mathbf{adb}$ and the new edges are IH-edges because $\mathbf a$ and $\mathbf b$ correspond to half-integer vertices and $\mathbf d$ corresponds to an I-vertex of the red Venkov graph.}
\label{pict:d_edge}
\end{figure}
\end{center}

After performing these steps, the new representative of $x+\mathcal S$ contains only S-, H- or IH-edges, and does not contain D-vertices.

\subsection{The cycle $x$ contain H-edges} Suppose $\mathbf{ab}$ is an H-edge of $x$. Then $\mathbf c:=\mathbf a \oplus \mathbf b$ is an integer vertex of $G(\mathcal D)$ and $\mathbf{abca}$ is a half-belt cycle. Moreover, since $\mathbf{ab}$ is not a D-edge, then $\mathbf a \oplus \mathbf b$ is not a D-vertex.

Swapping the edge $\mathbf{ab}$ with the pair of edges $\mathbf{acb}$ does not change the coset $x+\mathcal S$ and changes an H-edge to two IH-edges. After that our cycle $x$ will contain only IH- and S-edges and all integer vertices of $x$ are I-vertices.

\subsection{The cycle $x$ contains I-vertices} Suppose $\mathbf a$ is an I-vertex of $x$. Since there are no II-edges in $x$, there are half-integer vertices $\mathbf b$ and $\mathbf c$ such that $x$ contains the pair of edges $\mathbf{bac}$.

According to Lemma \ref{lem:i_vertex}, $\mathbf b$ and $\mathbf c$ coincide with $\frac 12\mathbf a$ in all non-zero coordinates of $\mathbf a$. We can connect $\mathbf b$ and $\mathbf c$ with a path of S-edges such that every vertex of this path is connected with $\mathbf a$. We claim that we can change the pair of edges $\mathbf{bac}$ with this path without changing the coset.

If $\mathbf a$ has only one non-zero coordinate, then the proof is given in the proof of Lemma \ref{lem:d_even}. Otherwise, similarly to the proof of Lemma \ref{lem:d_even}, it is enough to to treat only the case when $\mathbf b$ and $\mathbf c$ are connected with an S-edge.

Let $\mathbf f :=\mathbf b \oplus \mathbf c$. The vertex $\mathbf f$ of $G(\mathcal D)$ is an integer vertex and the corresponding vector has exactly one non-zero coordinate. In that case we can find a free sum $P$ of two cubes in the Delone decomposition of $\ssd_{2m}^*$ such that $\mathbf a$ represents a diagonal of some face of one cube, $\mathbf f$ represents a side of the other cube, and $\mathbf b$ and $\mathbf c$ represent two vectors between the cubes. This is true because the non-zero coordinate of $\mathbf f$ is on the position where $\mathbf a$ has zero coordinate.

%


Moreover, we may choose $P$ in such a way that the it contains the origin $\mathbf 0$ and two points $\mathbf b$ and $\mathbf c$. Let $F$ be the smallest face of $P$ that contains diagonal corresponding to $\mathbf a$ and the origin $\mathbf 0$ is one of the vertices of this diagonal. Then the convex hull of $F\cup \{\mathbf{b}\}\cup \{\mathbf{c}\}$ is a face of $P$ and must be subdivided in $\mathcal D$. This subdivision will contain the tetrahedron $\mathbf{0abc}$ and the cycle $\mathbf{bacb}$ is a trivially contractible cycle in this tetrahedron. Thus we can swap two edges $\mathbf{bac}$ with the edge $\mathbf{bc}$. 

Repeating this approach while $x$ has at least one I-vertex, we get a cycle from coset $x+\mathcal S$ with only S-edges.

\subsection{Concluding steps for the cycle $x$ with only S-edges} Similarly to Lemma~\ref{lem:d_even}, we can represent $x$ (or its rational multiple) as a combination of cycles comprising two-dimensional faces of the cube $C_{2m}$. However, for the Delone subdivision $\mathcal D$, not all such cycles will be trivially contractible because some of the pyramids described in the proof of Lemma \ref{lem:d_even} could be subdivided into pairs of tetrahedra. Nevertheless, the cycle composed of edges of such subdivided pyramid incident to its apex will be a combination of two trivially contractible cycles around the same vertex in the two new tetrahedra. Thus, every cycle in the one-dimensional skeleton of $C_{2m}$ belongs to $\mathcal S$ and the group $\mathcal G$ is generated by basic cycles in this case as well.
\end{document}